\documentclass[11pt]{amsart}
\usepackage[margin=1in]{geometry}

\usepackage{amssymb,comment}
\usepackage{amsthm}
\usepackage{amsmath}
\usepackage{mathrsfs}
\usepackage{amsbsy}
\usepackage[all]{xy}
\usepackage{bm}
\usepackage{hyperref}
\usepackage{tikz}
\usepackage{array}
\usepackage{float}
\usepackage{enumerate}
\usepackage{xcolor}
\usepackage{hhline}
\allowdisplaybreaks
\usepackage[noadjust]{cite}

\usepackage{caption}
\usepackage{subcaption}
\usepackage{tabu}
\usepackage{diagbox}
\usepackage{tikz}
\usepackage{bbm}
\usepackage{booktabs}

\DeclareFontFamily{U}{rcjhbltx}{}
\DeclareFontShape{U}{rcjhbltx}{m}{n}{<->rcjhbltx}{}
\DeclareSymbolFont{hebrewletters}{U}{rcjhbltx}{m}{n}
\DeclareMathOperator{\ptn}{part}
\DeclareMathOperator{\bu}{{\bf u}}
\DeclareMathOperator{\FS}{FS}
\DeclareMathOperator{\Vol}{vol}

\let\aleph\relax\let\beth\relax

\DeclareMathSymbol{\aleph}{\mathord}{hebrewletters}{39}
\DeclareMathSymbol{\beth}{\mathord}{hebrewletters}{98}

\usepackage[noabbrev,capitalise]{cleveref}

\newenvironment{enumerate*}%
  {\begin{enumerate}[(I)]%
    \setlength{\itemsep}{10pt}%
    \setlength{\parskip}{0pt}}%
  {\end{enumerate}}

\newtheorem{theorem}{Theorem}[section]
\newtheorem{proposition}[theorem]{Proposition}

\newtheorem{lemma}[theorem]{Lemma}

\theoremstyle{definition}

\newcommand{\Z}{\mathbb{Z}}
\title{Gap sets of random generalized numerical semigroups}

\author{Veronica Bitonti}
\address{Veronica Bitonti, Merton College, Oxford and Mathematical Institute, University of Oxford; Merton Street, Oxford OX1 4JD, United Kingdom}
\email{veronica.bitonti@maths.ox.ac.uk}

\author{Noah Kravitz}
\address{Noah Kravitz, St John's College, Oxford and Mathematical Institute, University of Oxford; St Giles', Oxford OX1 3JP, United Kingdom}
\email{noah.kravitz@maths.ox.ac.uk}

\begin{document}

\begin{abstract}
For a fixed positive integer $d$ and a small real $p>0$, sample a $p$-random subset $A \subseteq \mathbb{Z}_{\geq 0}^d$, and let $S:=\langle A \rangle$ be the generalized numerical semigroup generated by $A$.  We show that with high probability (as $p \to 0$), the gap set $\mathbb{Z}_{\geq 0}^d \setminus S$ is well approximated by the shifted hyperboloid region $$\{(x_1, \ldots, x_d) \in \mathbb{R}_{\geq 0}^d: (x_1+\log p^{-1}) \cdots (x_d+\log p^{-1})\ll p^{-1}(\log p^{-1})^{d+1}\}.$$
This generalizes work of the second author, Morales, and Schildkraut on the $1$-dimensional setting.  We also obtain the same result with $S$ replaced by the set of subset sums of $A$.
\end{abstract}

\maketitle

\section{Introduction}

\subsection{Random numerical semigroups}
A subset $S$ of an abelian group $G$ is a \emph{semigroup} if it is closed under addition.  Much attention has been lavished on the setting $G=\mathbb{Z}$, where a cofinite semigroup $S \subseteq \mathbb{Z}_{\geq 0}$ containing $0$ is called a \emph{numerical semigroup}; the complement $\mathbb{Z}_{\geq 0} \setminus S$ is called the \emph{gap set} of $S$.  Classical invariants attached to a numerical semigroup $S$ include the \emph{Frobenius number} (the largest element of the gap set) and the \emph{genus} (the size of the gap set).%, and the \emph{embedding dimension} (the minimum size of a generating set for $S$).

A popular topic of inquiry has been the behavior of ``random'' numerical semigroups.  One such notion is the \emph{Erd\H{o}s--R\'enyi model}, first introduced by De Loera, O'Neill, and Wilburne~\cite{DLOW}.  For any subset $A \subseteq \mathbb{Z}_{\geq 0}$, the set
$$\langle A \rangle:=\{a_1+\cdots+a_k: k \in \mathbb{Z}_{\geq 0}, ~a_1, \ldots, a_k \in A\}$$
generated by $A$ is a semigroup (the summands $a_i$ may be repeated); it is a numerical semigroup if and only if the total gcd of the elements of $A$ is $1$.  %(It is easy to see that every numerical semigroup has a unique minimal generating set and that this minimal generating set is finite.)  
For the Erd\H{o}s--R\'enyi model, fix a small parameter $p>0$ and sample a $p$-random  subset $A \subseteq \mathbb{Z}_{\geq 0}$, namely, a subset in which each positive integer is included independently with probability $p$; then form the semigroup $S:=\langle A \rangle$.  Notice that $S$ is a bona fide numerical semigroup with probability $1$ since $A$ will contain $2$ consecutive integers.  Typically $A$ is not the minimal generating set for $S$.  See \cite{Aliev2011} and the references therein for work on other models of random numerical semigroups.

Improving earlier results of~\cite{DLOW, BM}, the second author, Morales, and Schildkraut~\cite{KMS} established the with-high-probability asymptotic behavior of the aforementioned two invariants for Erd\H{o}s--R\'enyi random numerical semigroups.

\begin{theorem}[{\cite[Theorem 1.1]{KMS}}]\label{thm:1-dim}
Let $p>0$, sample a $p$-random subset $A \subseteq \mathbb{Z}_{ \geq 0}$, and set $S:=\langle A \rangle$.  Then with high probability (as $p \to 0$), the genus and Frobenius number of $S$ are both of size $\asymp p^{-1}(\log p^{-1})^2$.%, and the embedding dimension of $S$ is of size $\asymp (\log p^{-1})^2$.
\end{theorem}

The paper~\cite{KMS} (see Propositions 2.5 and 3.3 there) actually produces an absolute constant $C>0$ such that with high probability, $S$ contains a $o(1)$-proportion of the positive integers up to  $C^{-1}p^{-1}(\log p^{-1})^2$ and contains all of the positive integers larger than $Cp^{-1}(\log p^{-1})^2$; thus the gap set $\mathbb{Z}_{\geq 0} \setminus S$ is well approximated by the interval of positive integers up to scale $p^{-1}(\log p^{-1})^2$.

\subsection{Random generalized numerical semigroups}

It is also natural to look at higher-dimensional analogues of this problem.  In the setting $G=\mathbb{Z}^d$ with $d>1$, a cofinite semigroup $S \subseteq \mathbb{Z}_{\geq 0}^d$ containing $0$ is often called a \emph{generalized numerical semigroup}.  In general there is not a single ``largest'' gap element that can be used to define the Frobenius number of a generalized numerical semigroup, but the genus (number of gaps) is still a useful concept.  The Erd\H{o}s--R\'enyi model in the multidimensional setting is defined just as in the $1$-dimensional setting.

Our main result provides a with-high-probability description of the gap set of an Erd\H{o}s--R\'enyi random generalized numerical semigroup.  The interesting new aspect of the $d>1$ setting is that one can ask about not only the \emph{size} of the gap set $\mathbb{Z}_{\geq 0}^d \setminus S$ but also its \emph{shape}.  To this end, let $$\mathcal{R}_d(p,Z):=\{(x_1, \ldots, x_d) \in \mathbb{Z}_{\geq 0}^d: (x_1+\log p^{-1}) \cdots (x_d+\log p^{-1})\leq Z\}$$
denote the set of lattice points under a shifted hyperboloid; see \cref{fig:R2-region}.  We show that this shifted hyperboloid is, in an appropriate sense, the ``scaling limit'' for the gap set of a random $\langle A \rangle$.

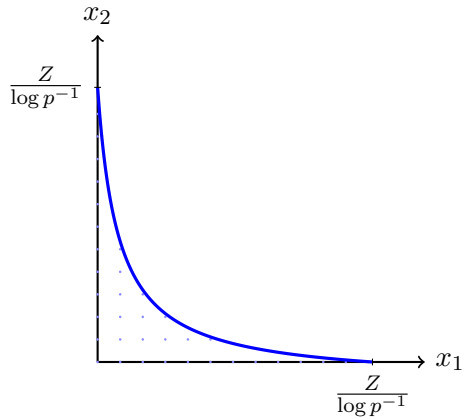
\begin{figure}
    \centering
\begin{tikzpicture}[scale=0.30]

    % Parameters: change these as desired
    \pgfmathsetmacro{\p}{0.4}
    \pgfmathsetmacro{\Z}{12}
    \pgfmathsetmacro{\a}{ln(1/\p)}   % a = log(p^{-1})

    % Useful endpoint: where the curve meets the x_1-axis
    \pgfmathsetmacro{\xmaxcurve}{\Z/\a - \a}
    \pgfmathsetmacro{\ymaxcurve}{\xmaxcurve}

    % Axis lengths
    \pgfmathsetmacro{\Xmax}{ceil(\xmaxcurve) + 1}
    \pgfmathsetmacro{\Ymax}{ceil(\ymaxcurve) + 1}

    % Axes
    \draw[->, thick] (0,0) -- (\Xmax+0.5,0) node[right] {$x_1$};
    \draw[->, thick] (0,0) -- (0,\Ymax+0.5) node[above] {$x_2$};

    % Optional tick marks
    %\foreach \x in {1,...,\Xmax}
       % \draw (\x,0.08) -- (\x,-0.08) node[below] {\small $\x$};
    %\foreach \y in {1,...,\Ymax}
       % \draw (0.08,\y) -- (-0.08,\y) node[left] {\small $\y$};

    % Lattice points in the region:
    % (x_1 + a)(x_2 + a) < Z, with x_1,x_2 in Z_{\ge 0}
    \foreach \i in {0,...,\Xmax} {
        \foreach \j in {0,...,\Ymax} {
            \pgfmathparse{(\i+\a)*(\j+\a) < \Z ? 1 : 0}
            \ifnum\pgfmathresult=1
                \fill[blue!45] (\i,\j) circle (1.6pt);
            \fi
        }
    }

    % Boundary curve: (x_1+a)(x_2+a)=Z  =>  x_2 = Z/(x_1+a)-a
    \draw[domain=0:\xmaxcurve, samples=200, smooth, very thick, blue]
        plot (\x, {\Z/(\x+\a) - \a});

    % Equation label
 %   \node[black, right] at ({0.15*\xmaxcurve}, {\Z/(0.15*\xmaxcurve+\a)-\a + 0.8})
%        {$\,Z=(x_1+\log p^{-1})(x_2+\log p^{-1})$};
        % Tick at intersection with x_1-axis
\draw (\xmaxcurve,0.15) -- (\xmaxcurve,-0.15)
    node[below] {$\frac{Z}{\log p^{-1}}$};
    % Tick at intersection with x_2-axis
\draw (0.15,\ymaxcurve) -- (-0.15,\ymaxcurve)
    node[left] {$\frac{Z}{\log p^{-1}}$};

\end{tikzpicture}
 \caption{The region $\mathcal{R}_2(p,Z)$ lies under the curve $(x_1+\log p^{-1})(x_2+\log p^{-1})=Z$.}
    \label{fig:R2-region}
\end{figure}

\begin{theorem}\label{thm:main}
For each positive integer $d$, there is a constant $C=C_d>0$ such that the following holds.  Let $p>0$, sample a $p$-random subset $A \subseteq \mathbb{Z}_{ \geq 0}^d$, and set $S:=\langle A \rangle$.  Then with high probability (as $p \to 0$), we have
$$
|S \cap \mathcal{R}_d(p,C^{-1} p^{-1}(\log p^{-1})^{d+1})|=o(|\mathcal{R}_d(p,C^{-1} p^{-1}(\log p^{-1})^{d+1})|)$$
and
$$\mathbb{Z}_{\geq 0}^d \setminus \mathcal{R}_d(p,C p^{-1}(\log p^{-1})^{d+1}) \subseteq S,
$$
and in particular the gap set has size $$|\mathbb{Z}_{\geq 0}^d \setminus S| \asymp_d p^{-1}(\log p^{-1})^{2d}.$$
%Moreover, with high probability the embedding dimension is $\asymp_d (\log p^{-1})^{2d}$.
\end{theorem}

We remark that by taking $C$ sufficiently large, one can get an arbitrary power-saving in the little-$o$ in the first part of the theorem (see the proof of Proposition~\ref{prop:lower} below).

Informally, Theorem~\ref{thm:main} says that the typical gap set $\mathbb{Z}_{ \geq 0}^d \setminus S$ is well approximated by the set of lattice points where each subset of $k$ coordinates (for $1 \leq k \leq d$) has product $\ll p^{-1}(\log p^{-1})^{k+1}$.  Near each coordinate hyperplane, one sees the typical behavior of a random generalized numerical semigroup in dimension $1$ smaller; away from the coordinate hyperplanes, the $k=d$ product constraint dominates.

The $d=1$ case of Theorem~\ref{thm:main} recovers Theorem~\ref{thm:1-dim}.  Conversely, one can apply Theorem~\ref{thm:1-dim} to each positive coordinate axis of $\mathbb{Z}_{\geq 0}^d$; this implies via the semigroup closure property that with high probability, a random numerical semigroup in $\mathbb{Z}_{\geq 0}^d$ contains the set of points where \emph{all} of the coordinates are at scale at least $p^{-1}(\log p^{-1})^2$.  This conclusion is significantly weaker than Theorem~\ref{thm:main}: For instance, in the regime where the coordinates $x_1, \ldots, x_d$ are all roughly equal to one another, Theorem~\ref{thm:main} tells us that the threshold value for inclusion in $S$ is at scale $p^{-1/d}(\log p^{-1})^{(d+1)/d}$, while the naive application of Theorem~\ref{thm:1-dim} tells us only that inclusion occurs past $p^{-1}(\log p^{-1})^2$.

The proof of Theorem~\ref{thm:main} both develops techniques from~\cite{KMS} and brings in new ideas.  We wish to emphasize that the upper and lower bounds match up to dilation by a constant factor; we found it surprising that our approach suffices to pin down the shape of the gap set so precisely.

\subsection{Random sets of subset sums}
Ben Green (personal communication) has suggested a variant of the above problem where the semigroup $S=\langle A \rangle$ is replaced by the set of subset sums
$$S^*=\FS(A):=\{a_1+\cdots+a_k: k \in \mathbb{Z}_{\geq 0}, a_1, \ldots, a_k \text{ distinct}\};$$
the difference between $\langle A \rangle$ and $\FS(A)$ is that the former allows repeated summands and the latter does not.  (The notation $\FS$ stands for ``finite sums''.)  The trivial deterministic containment $\FS(A) \subseteq \langle A \rangle$ can be strict.  Nonetheless, we show that the with-high-probability shapes of the gap sets $\mathbb{Z}_{\geq 0}^d \setminus S$ and $\mathbb{Z}_{\geq 0}^d \setminus S^*$ for a $p$-random set $A \subseteq \mathbb{Z}_{\geq 0}^d$ are qualitatively the same.

\begin{theorem}\label{thm:subset-sums}
For each positive integer $d$, there is a constant $C=C_d>0$ such that the following holds.  Let $p>0$, sample a $p$-random subset $A \subseteq \mathbb{Z}_{ \geq 0}^d$, and set $S^*:=\FS(A)$.  Then with high probability (as $p \to 0$), we have
$$
|S^* \cap \mathcal{R}_d(p,C^{-1} p^{-1}(\log p^{-1})^{d+1})|=o(|\mathcal{R}_d(p,C^{-1} p^{-1}(\log p^{-1})^{d+1})|)$$
and
$$\mathbb{Z}_{\geq 0}^d \setminus \mathcal{R}_d(p,C p^{-1}(\log p^{-1})^{d+1}) \subseteq S^*,
$$
and in particular the gap set has size $$|\mathbb{Z}_{\geq 0}^d \setminus S^*| \asymp_d p^{-1}(\log p^{-1})^{2d}.$$
\end{theorem}

As in Theorem~\ref{thm:main}, we can obtain an arbitrary power-saving in the little-$o$ by taking $C$ sufficiently large.

\subsection{Structure of the paper}
The bulk of the paper consists of Sections~\ref{sec:lower} and~\ref{sec:upper}.  In the former section we prove that with high probability $S=\langle A \rangle$ is sparse in $$\mathcal{R}_d(p,C^{-1} p^{-1}(\log p^{-1})^{d+1});$$
our argument is an elaboration of the approach in \cite{KMS} with some new geometric ingredients and technical challenges.  In the latter section we prove that with high probability $S^*=\FS(A)$ contains $$\mathbb{Z}_{\geq 0}^d \setminus \mathcal{R}_d(p,C p^{-1}(\log p^{-1})^{d+1});$$
here we diverge more sharply from \cite{KMS} because $\FS(A)$ does not allow repeated summands.  We assemble the pieces to deduce Theorems~\ref{thm:main} and~\ref{thm:subset-sums} in Section~\ref{sec:assembling}.  Section~\ref{sec:concluding} contains a few remarks and open problems. In Appendix~\ref{sec:appendix} we calculate the asymptotic size of $\mathcal{R}_d(p,Z)$.

\subsection{Notation}
Throughout the paper, $0<p<1$ is the density parameter of a random set.  We say that an event occurs \emph{with high probability} if its probability of occurring tends to $1$ as $p \to 0$; when necessary, we will freely assume that $p$ is sufficiently small.

We use the following standard asymptotic notation.  We write $f \ll g$ or $f=O(g)$ or $g=\Omega(f)$ if there is an absolute constant $C>0$ such that $|f| \leq Cg$.  We write $f \asymp g$ if $f \ll g$ and $g \ll f$.  We write $f=o(g)$ if $f/g \to 0$ in the relevant limit (as the density parameter $p$ tends to $0$, unless explicitly noted otherwise).  Parameters subscripted to asymptotic notation indicate that the constant $C$ may depend on these parameters (but no others); for instance, $f \ll_d g$ means that there is a constant $C=C(d)>0$ depending only on $d$ such that $f \leq Cg$.

For integers $M \leq N$, we write $[M,N]:=\{M, M+1, \ldots, N\}$ for the discrete interval between $M$ and $N$.  Logarithms are base-$e$ unless otherwise noted.

\section{Lower bounds}\label{sec:lower}

In this section we prove the ``sparsity'' part of Theorem~\ref{thm:main}.

\begin{proposition}\label{prop:lower}
For each positive integer $d$, there is a constant $c=c_d>0$ such that the following holds.  Let $p>0$, sample a $p$-random subset $A \subseteq \mathbb{Z}_{\geq 0}^d$, and set $S:=\langle A \rangle$.  Then with high probability we have
$$|S \cap \mathcal{R}_d(p,c p^{-1}(\log p^{-1})^{d+1})|=o(|\mathcal{R}_d(p,c p^{-1}(\log p^{-1})^{d+1})|).$$
\end{proposition}

We start with the observation that the continuous region
$$\{(x_1, \ldots, x_d) \in \mathbb{R}_{\geq 0}^d: (x_1+\log p^{-1}) \cdots (x_d+\log p^{-1})\geq Z\},$$
corresponding to the complement of $\mathcal{R}_d(p,Z)$ in the positive quadrant, is convex.  Every closed convex subset of $\mathbb{R}^d$ is the intersection of its supporting half-spaces and thus can be well approximated by the intersection of a fine finite net of supporting half-spaces.  In our context, this means that $\mathcal{R}_d(p,Z)$ can be well approximated by the union of the regions under a fine net of tangent hyperplanes (or, more precisely, low-lying secant hyperplanes) to the surface $$\{(x_1, \ldots, x_d) \in \mathbb{R}_{\geq 0}^d: (x_1+\log p^{-1}) \cdots (x_d+\log p^{-1})=Z\}.$$
%See \cref{fig:tangent-approximation} for an illustration in $2$ dimensions.

We will use well-known asymptotics for the number of ``colored'' integer partitions to show that with very high probability, $S$ contains few elements below each supporting hyperplane in our net; we will then deduce Proposition~\ref{prop:lower} by union-bounding over the hyperplanes.  More precisely, given a vector $\lambda \in \mathbb{R}_{>0}^d$, we will consider the lattice points whose dot products with $\lambda$ lie below a suitable threshold, and we will show that few such points lie in $S$.  We will see that our approach weakens when $\lambda$ has coordinates very close to $0$; this gives one explanation for the appearance of the quantities $x_i+\log p^{-1}$ (rather than just $x_i$) in the definition of $\mathcal{R}_d(p,Z)$.

\subsection{Asymptotics for colored integer partitions}
In the treatment~\cite{KMS} of the $1$-dimensional case, an important ingredient was the classical integer partition asymptotic due to Hardy and Ramanjuan~\cite{HR} and to Uspensky~\cite{Usp}.  For our multidimensional setting, we will use a more general statement about ``colored'' integer partitions due to Meinardus~\cite{Mei}.

For a sequence $\bu=(u_1, u_2, \ldots)$ of nonnegative integer weights, let $\ptn_{\bu}(n)$ denote the number of integer partitions of $n$ where the parts of size $k$ come in $u_k$ different ``colors''.  The case $\bu=(1,1,\ldots)$ corresponds to the usual integer partition function.  The numbers $\ptn_{\bu}(n)$ are characterized by the generating function
$$\sum_{n=0}^\infty \ptn_{\bu}(n) z^n=\prod_{k=1}^\infty \frac{1}{(1-z^k)^{u_k}}.$$
Meinardus~\cite{Mei} has used the saddle-point method to obtain asymptotics for $\ptn_{\bu}(n)$ when the weights $\bu$ grow as $u_k=k^{O(1)}$; see~\cite[Chapter 6]{And} for an exposition and further discussion.

We are concerned with the special case where $u_k=k^{d-1}$ for $d$ a positive integer; for brevity, write $\ptn_d(n):=\ptn_{(1^{d-1}, 2^{d-1}, \ldots)}(n)$.  Here, Meinardus's main result~\cite[Satz 1]{Mei} implies the asymptotic
\begin{equation}\label{eq:ptn}
\ptn_d(n)=\exp((1+o_d(1))\kappa_d n^{d/(d+1)})
\end{equation}
for an explicit constant $\kappa_d>0$.  Meinardus in fact provides a more precise asymptotic for $\ptn_d(n)$ (up to a multiplicative factor of $1+n^{-\Omega_d(1)}$), but all that matters for us is the value of the exponent $d/(d+1)$ in \eqref{eq:ptn}.  We deploy Meinardus's result in the form of the following lemma.  %For our applications it matters only that we obtain a bound of the shape $\exp(O_d((\gamma^{1/d}Y)^{d/(d+1)}))$.

\begin{lemma}\label{lem:ptn}
Let $d$ be a positive integer.  Let $\gamma,Y>0$, and let $\lambda \in \mathbb{R}_{>0}^d$ be any positive vector.  Suppose that $A_0\subseteq \mathbb{Z}_{\geq 0}^d$ satisfies
$$|\{x \in A_0: x \cdot \lambda \leq y\}| \leq \gamma y^{d}$$
for all real numbers $y \geq 0$.  Then
$$|\langle A_0 \rangle \cap \{x \in \mathbb{Z}_{\geq 0}^d: x \cdot \lambda \leq Y\}|\leq \exp((1+o_d(1))\kappa_d (2\gamma^{1/d}Y)^{d/(d+1)})$$
in the regime $\gamma^{1/d}Y \to \infty$, where $\kappa_d$ is the constant from \eqref{eq:ptn}.
\end{lemma}

\begin{proof}
Order the elements of $A_0$ as $a_1, a_2, \ldots$ such that the dot products $a_i \cdot \lambda$ are non-decreasing as $i$ increases (breaking ties arbitrarily).  For each $i$, our assumption with $y=a_i \cdot \lambda$ reads
$$i \leq |\{x \in A_0: x \cdot \lambda \leq a_i \cdot \lambda\}| \leq \gamma (a_i \cdot \lambda)^{d},$$
whence we have
$$a_i \cdot \lambda \geq (\gamma^{-1}i)^{1/d}.$$
Now, each element of $\langle A_0 \rangle$ can be expressed (possibly non-uniquely) as
$$x=b_1a_1+b_2a_2+\cdots$$
where $(b_1, b_2, \ldots)$ is a tuple of nonnegative integers (with only finitely many nonzero entries).  The dot product of such an $x$ with $\lambda$ is
$$x \cdot \lambda=\sum_i b_i(a_i \cdot \lambda) \geq \sum_i b_i (\gamma^{-1}i)^{1/d},$$
so $|\langle A_0 \rangle \cap \{x \in \mathbb{Z}_{\geq 0}^d: x \cdot \lambda \leq Y\}|$ is upper-bounded by the number of tuples $(b_1, b_2, \ldots)$ satisfying
$$\sum_i b_i (\gamma^{-1}i)^{1/d} \leq Y.$$
Rearranging, we see that any tuple satisfying this condition also satisfies
$$\sum_i b_i \lfloor 2i^{1/d} \rfloor  \leq 2\gamma^{1/d}Y.$$
For each positive integer $k$, we have $\lfloor 2i^{1/d} \rfloor=k$ if and only if $2^{-d}k^d \leq i <2^{-d}(k+1)^d$; the length of this interval is trivially at most $k^{d-1}$, so the number of $i$'s with $\lfloor 2i^{1/d} \rfloor=k$ is at most $k^{d-1}$.

Thus each tuple $(b_1, b_2, \ldots)$ can be interpreted as an integer partition of some $n \leq 2\gamma^{1/d}Y$, where the parts of size $k$ come in at most $k^{d-1}$ colors.  By Meinardus's theorem \eqref{eq:ptn}, the total number of possible tuples is at most
$$\sum_{n=0}^{2\gamma^{1/d}Y}\ptn_d(n)=\sum_{n=0}^{2\gamma^{1/d}Y}\exp((1+o_d(1))\kappa_d n^{d/(d+1)})=\exp((1+o_d(1))\kappa_d (2\gamma^{1/d}Y)^{d/(d+1)}).$$
This completes the proof of the lemma.
\end{proof}

\subsection{Sparseness under a hyperplane}

Recall that our strategy for Proposition~\ref{prop:lower} involves estimating the region $\mathcal{R}_d(p,Z)$ by the union of the tetrahedral regions under several hyperplanes and then showing that with very high probability $S=\langle A \rangle$ is sparse in each such tetrahedral region.  It will be convenient to partition $A$ according to which coordinates vanish and then to bound the contribution from each part separately.  The following technical result underpins our approach.

\begin{proposition}\label{prop:region-under-hyperplane}
For each positive integer $d$ and every real $\varepsilon>0$, there is a constant $c=c_d(\varepsilon)>0$ such that the following holds.  Let $p>0$, and suppose that $x_1, \ldots, x_d>0$ satisfy $$x_1 \cdots x_d \leq c p^{-1}(\log p^{-1})^{d+1}.$$
Sample a $p$-random subset $A \subseteq \mathbb{Z}_{>0}^d$, and set $S:=\langle A \rangle$.  Then with probability at least $$1-e^{-c(\log p^{-1})^{1/2}},$$
we have
$$\left|S \cap \left\{(y_1, \ldots, y_d) \in \mathbb{Z}_{\geq 0}^d: \frac{y_1}{x_1}+\cdots+\frac{y_d}{x_d} \leq 1\right\}\right| \leq p^{-\varepsilon}.$$
\end{proposition}

\begin{proof}
We start with some setup.  By possibly increasing some of the $x_i$'s, we may assume that their product $x_1 \cdots x_d$ is equal to exactly $$Z:=c p^{-1}(\log p^{-1})^{d+1},$$
where $c>0$ is a constant to be determined later.
For each real $h \geq 0$, define the sets
$$T^*(h):=\left\{(y_1, \ldots, y_d) \in \mathbb{Z}_{\geq 0}^d: \frac{y_1}{x_1}+\cdots+\frac{y_d}{x_d} \leq h\right\}$$
and
$$T(h):=\left\{(y_1, \ldots, y_d) \in \mathbb{Z}_{>0}^d: \frac{y_1}{x_1}+\cdots+\frac{y_d}{x_d} \leq h\right\};$$
these differ in that the former allows zero coordinates but the latter does not.  Our goal is to show that $|S \cap T^*(1)|$ is very small with very high probability.  Notice that $A \cap T^*(h)=A \cap T(h)$ due our assumption $A \subseteq \mathbb{Z}_{>0}^d$.  If we identify each point $(y_1, \ldots, y_d) \in T(h)$ with the unit cube $$\{(z_1, \ldots, z_d) \in \mathbb{R}^d: \lceil z_i \rceil=y_i ~\forall i\},$$
then these cubes are disjoint and their union is contained in the (continuous) tetrahedron
$$\widetilde T(h):=\left\{(y_1, \ldots, y_d) \in \mathbb{R}_{>0}^d: \frac{y_1}{x_1}+\cdots+\frac{y_d}{x_d} \leq h\right\}.$$
Thus we obtain the upper bound
$$|T(h)| \leq \Vol(\widetilde T(h))=\frac{(x_1 \cdots x_d) h^d}{d!}=\frac{Zh^d}{d!};$$
we emphasize that such an estimate is not available for $T^*(h)$.

Notice that the elements of $A \setminus T(1)$ do not contribute to $S \cap T^*(1)$.  We partition $$A \cap T(1)=A_1 \sqcup A_2 \sqcup A_3$$ into three parts by defining
$$A_1:=A \cap T(h_1), \quad A_2:=A \cap (T(h_2) \setminus T(h_1)), \quad \text{and} \quad A_3:=A \cap (T(1) \setminus T(h_2)),$$
for some parameters $0<h_1<h_2\leq 1$ to be specified later.  We will show that with very high probability, $A_1$  is empty and each of $A_2,A_3$ contributes few (at most $p^{-\varepsilon/2}$) elements of $\langle A \rangle \cap T^*(1)$.

First, since $\mathbb{E} [|A_1|] \leq p|T(h_1)| \leq pZh_1^d/d!$, Markov's inequality gives
\begin{equation*}\label{eq:A_2}
\mathbb{P}[A_1=\emptyset] \geq 1-\frac{pZh_1^d}{d!}.
\end{equation*}

Second, since $|T(h_2) \setminus T(h_1)| \leq |T(h_2)| \leq Zh_2^d/d!$, a Chernoff bound gives
$$\mathbb{P}\left[|A_2| \geq \frac{2pZh_2^d}{d!}\right] \leq \exp\left(-\frac{pZh_2^d}{3\cdot d!}\right).$$
Combining this observation with the trivial estimate
$$|\langle A_2 \rangle \cap T^*(1)| \leq (h_1^{-1}+1){|A_2|}$$
(note that in any element of $\langle A_2 \rangle \cap T^*(1)$, each element of $A_2$ appears as a summand with multiplicity at most $h_1$), we find that
\begin{equation*}\label{eq:A_1}
\mathbb{P}\left[ |\langle A_2 \rangle \cap T^*(1)|\leq (h_1^{-1}+1)^{2pZh_2^d/d!} \right] \geq 1-\exp\left(-\frac{pZh_2^d}{3\cdot d!}\right).
\end{equation*}

Third, to control $|\langle A_3 \rangle \cap T^*(1)|$, we wish to use Lemma~\ref{lem:ptn} with $\lambda:=(1/x_1, \ldots, 1/x_d)$ and $Y:=1$ and a suitable $\gamma>0$ (tending to infinity as $p \to 0$).  The hypothesis of that lemma is that $$|A_3 \cap T^*(h)|=|A_3 \cap T(h)| \leq \gamma h^d$$ for all $h \geq 0$ (here where we make critical use of the assumption that $A \subseteq \mathbb{Z}_{>0}^d$).  The definition of $A_3$ ensures that this is unconditionally the case in the range $h \leq h_2$.  For the range $h>h_2$, it suffices to check that with very high probability
$$|A_3 \cap T(2^{\ell+1} h_2)| \leq \gamma (2^\ell h_2)^d$$
for all $\ell \in \mathbb{Z}_{\geq 0}$.  For each such $\ell$, a Chernoff bound gives
$$\mathbb{P}\left[|A_3 \cap T(2^{\ell+1}h_2)| \geq \frac{2pZ(2^{\ell+1}h_2)^d}{d!}\right] \leq \exp\left(-\frac{pZ(2^{\ell+1}h_2)^d}{3\cdot d!}\right).$$
A union bound over $\ell$ tells us that with probability at least
$$1-\sum_{\ell=0}^\infty \exp\left(-\frac{pZ(2^{\ell+1}h_2)^d}{3\cdot d!}\right),$$
we have
$$|A_3 \cap T(2^{\ell+1}h_2)| \leq \frac{2pZ(2^{\ell+1}h_2)^d}{d!}$$
simultaneously for all $\ell \in \mathbb{Z}_{\geq 0}$.  In such an outcome, applying Lemma~\ref{lem:ptn} with
$$\gamma:=\frac{4pZ}{d!}$$
and $Y:=1$ gives
$$|\langle A_3 \rangle \cap T^*(1)| \leq \exp((1+o_d(1))\kappa_d (2\gamma^{1/d})^{d/(d+1)})=\exp(O_d((pZ)^{1/(d+1)})).$$

It remains to assemble the pieces.  Notice that $|\langle A \rangle \cap T^*(1)| \leq \prod_{i=1}^3 |\langle A_i \rangle \cap T^*(1)|$.  Thus the previous three paragraphs together give (via a union bound) that with probability at least
\begin{equation}\label{eq:under-hyperplane-prob}
1-\frac{pZh_1^d}{d!}-\exp\left(-\frac{pZh_2^d}{3\cdot d!}\right)-\sum_{\ell=0}^\infty \exp\left(-\frac{pZ(2^{\ell+1}h_2)^d}{3\cdot d!}\right),
\end{equation}
we have
\begin{equation}\label{eq:under-hyperplane-size}
|\langle A \rangle \cap T^*(1)| \leq (h_1^{-1}+1)^{2pZh_2^d/d!} \cdot \exp(O_d((pZ)^{1/(d+1)})).
\end{equation}
We now specify our parameters.  Take $c>0$ small enough that the argument $O_d((pZ)^{1/(d+1)})$ of the preceding exponential is at most $(\varepsilon/2)\log p^{-1}$.  Then set
$$h_1:=\exp\left(-\frac{\varepsilon^{1/2} \cdot d!}{5c}  (\log p^{-1})^{1/2} \right) \quad \text{and} \quad h_2:=\left(\frac{\varepsilon}{(\log p^{-1})^{2d+1}} \right)^{1/2d}.$$
With this choice, each of the first two subtracted terms in \eqref{eq:under-hyperplane-prob} is safely
$$\exp(-\Omega_{d,\varepsilon}((\log p^{-1})^{1/2}));$$
as long as $p$ is sufficiently small (relative to $c$), the sum in \eqref{eq:under-hyperplane-prob} is dominated by the $\ell=0$ contribution, which is likewise $\exp(-\Omega_{d,\varepsilon}((\log p^{-1})^{1/2}))$.  Finally, for \eqref{eq:under-hyperplane-size}, each term on the right-hand side is at most $p^{\varepsilon/2}$.  This concludes the proof.
\end{proof}

We now incorporate the contributions from the parts of $A$ with some coordinates vanishing.  It is here that the condition $x_i \gg \log p^{-1}$ comes into the picture.

\begin{proposition}\label{prop:region-under-hyperplane-with-zeros}
For each positive integer $d$ and every real $\varepsilon>0$, there is a constant $c=c_d(\varepsilon)>0$ such that the following holds.  Let $p>0$, and suppose that $x_1, \ldots, x_d\geq \log p^{-1}$ satisfy $$x_1 \cdots x_d \leq c p^{-1}(\log p^{-1})^{d+1}.$$
Sample a $p$-random subset $A \subseteq \mathbb{Z}_{\geq 0}^d$, and set $S:=\langle A \rangle$.  Then with probability at least $$1-e^{-c(\log p^{-1})^{1/2}},$$
we have
$$\left|S \cap \left\{(y_1, \ldots, y_d) \in \mathbb{Z}_{\geq 0}^d: \frac{y_1}{x_1}+\cdots+\frac{y_d}{x_d} \leq 1\right\}\right| \leq p^{-\varepsilon}.$$
\end{proposition}

\begin{proof}
Following the notation from the preceding proof, define the tetrahedral region $$T^*:=\left\{(y_1, \ldots, y_d) \in \mathbb{Z}_{\geq0}^d: \frac{y_1}{x_1}+\cdots+\frac{y_d}{x_d} \leq 1\right\}.$$  For each subset $I \subseteq [d]$, let $A_I$ consist of the points $(x_1, \ldots, x_d) \in A$ such that
$$\{i \in [d]: x_i\neq 0\}=I.$$
Then we have a partition $A=\sqcup_{I \subseteq [d]}A_I$ and we get the corresponding upper bound
$$|\langle A \rangle \cap T^*| \leq \prod_{I \subseteq [d]} |\langle A_I \rangle \cap T^*|.$$
The contribution from $I=\emptyset$ is $1$.  To complete the proof, we will show that for each nonempty $I \subseteq [d]$, we have
$$|\langle A_I \rangle \cap T^*| \leq p^{-\varepsilon/2^d}$$
with probability at least $1-\exp(-\Omega_{d,\varepsilon}((\log p^{-1})^{1/2}))$.

Fix some nonempty $I \subseteq [d]$.  By permuting the coordinates, we may assume that $I=[m]$ for some $1 \leq m \leq d$ (this merely simplifies notation).  Let us restrict our attention to the copy of $\mathbb{Z}^m$ corresponding to the first $m$ coordinates.  Here, $A_I$ is a $p$-random subset of $\mathbb{Z}_{>0}^m$, and our tetrahedral region $T^*$ corresponds to the new region
$$T_I:=\left\{(y_1, \ldots, y_m) \in \mathbb{Z}_{>0}^d: \frac{y_1}{x_1}+\cdots+\frac{y_m}{x_m} \leq 1\right\}.$$
Our assumption $x_{m+1}, \ldots, x_d \geq \log p^{-1}$ guarantees that
$$x_1 \cdots x_m \leq cp^{-1}(\log p^{-1})^{m+1}.$$
Let us now apply Proposition~\ref{prop:region-under-hyperplane} with $d$ replaced by $m$ and with $\varepsilon$ replaced by $\varepsilon/2^d$.  As long as $c>0$ is sufficiently small, this proposition tells us that with probability at least $1-\exp(-\Omega_{d,\varepsilon}((\log p^{-1})^{1/2}))$, we have
$$|\langle A_I \rangle \cap T_I| \leq p^{-\varepsilon/2^d}.$$
This produces the desired conclusion.
\end{proof}

We remark that the quasi-exponential bound on the error probability is stronger than what we will need in our application, where sufficiently good polylogarithmic savings would suffice.

\subsection{Covering the hyperboloid region}
To deduce Proposition~\ref{prop:lower}, we cover the region $$\mathcal{R}_d(p,cp^{-1}(\log p^{-1})^{d+1})$$ with a small number of regions of the form treated in Proposition~\ref{prop:region-under-hyperplane-with-zeros}.

\begin{lemma}\label{lem:covering with hyperplanes}
Let $d$ be a positive integer, and let $p,Z>0$ be real parameters.  Then
$$\mathcal{R}_d(p,Z) \subseteq \bigcup_{(e_1, \ldots, e_d) \in \mathcal{E}}\left\{ (y_1, \ldots, y_d) \in \mathbb{Z}_{\geq 0}^d: \frac{y_1}{2^{e_1}}+\cdots+\frac{y_d}{2^{e_d}} \leq 1 \right\},$$
where we have defined the index set
$$\mathcal{E}:=\{(e_1, \ldots, e_d) \in \mathbb{Z}^d: e_i \geq \log_2 \log p^{-1} ~\forall i \quad \text{and} \quad  e_1+\cdots+e_d \leq d+d\log_2 d+\log_2 Z\}.$$
\end{lemma}

\begin{proof}
Let $(x_1, \ldots, x_d) \in \mathcal{R}_d(p,Z)$.  Then each $x_i \geq 0$ and we have
$$(x_1+\log p^{-1}) \cdots (x_d+\log p^{-1}) \leq Z.$$
After possibly permuting the coordinates, we may assume that there is some $0 \leq m \leq d$ such that $x_1, \ldots, x_m \geq \log p^{-1}$ and $x_{m+1}, \ldots, x_d \leq \log p^{-1}$.  Define
$$e_i:=\begin{cases}
\lceil \log_2 x_i+\log_2 d \rceil, &\text{for } 1 \leq i \leq m;\\
\lceil \log_2 \log p^{-1}+\log_2 d \rceil, &\text{for } m<i \leq d.
\end{cases}$$
Clearly each $e_i \geq \log_2 \log p^{-1}$.  From the estimate
\begin{align*}
e_1+\cdots+e_d &\leq d+d\log_2 d+ \log_2(x_1\cdots x_m (\log p^{-1})^{d-m})\\
 &\leq d+d\log_2 d+\log_2((x_1+\log p^{-1}) \cdots (x_d+\log p^{-1}))\\
 &\leq d+d\log_2 d+\log_2 Z,
\end{align*}
we conclude that $(e_1, \ldots, e_d) \in \mathcal{E}$.  We will be done once we show that $$\frac{x_1}{2^{e_1}}+\cdots+\frac{x_d}{2^{e_d}} \leq 1.$$
Indeed, we have $2^{e_i} \geq dx_i$ for each $1 \leq i \leq d$, so each term in the sum is at most $1/d$, and the total is at most $1$.
\end{proof}

Proposition~\ref{prop:lower} follows quickly from the preceding lemma and Proposition~\ref{prop:region-under-hyperplane-with-zeros}.

\begin{proof}[Proof of Proposition~\ref{prop:lower}]
Let $\varepsilon>0$.  Let $c>0$ be the constant from Proposition~\ref{prop:region-under-hyperplane-with-zeros}.  Set $$Z:=2^{-d-d\log_2 d}cp^{-1}(\log p^{-1})^{d+1},$$ and let $\mathcal{E}$ be the set of indices produced by Lemma~\ref{lem:covering with hyperplanes}.  Notice that $|\mathcal{E}| \ll_{d,\varepsilon} (\log p^{-1})^d$.  For each $(e_1, \ldots, e_d) \in \mathcal{E}$, the tuple $(x_1, \ldots, x_d):=(2^{e_1}, \ldots, 2^{e_d})$ satisfies the hypotheses of Proposition~\ref{prop:region-under-hyperplane-with-zeros}.  This proposition and a union bound over $\mathcal{E}$ tell us that with probability at least $$1-\exp(-\Omega_{d,\varepsilon}((\log p^{-1})^{1/2}))$$ (notice that the polylogarithmic size of $\mathcal{E}$ is a lower-order contribution), we have
$$\left|\langle A \rangle \cap \left\{ (y_1, \ldots, y_d) \in \mathbb{Z}_{\geq 0}^d: \frac{y_1}{2^{e_1}}+\cdots+\frac{y_d}{2^{e_d}} \leq 1 \right\}\right| \leq p^{\varepsilon}$$
for all $(e_1, \ldots, e_d) \in \mathcal{E}$.  In such an outcome, it follows from Lemma~\ref{lem:covering with hyperplanes} that
$$|\langle A \rangle \cap \mathcal{R}_d(p,Z)| \leq p^{\varepsilon.}$$
We compute in Lemma~\ref{lem:appendix} that $|\mathcal{R}_d(p,Z)| \asymp_{d,c} p^{-1} (\log p^{-1})^{d+1}$.  Taking $\varepsilon:=1/2$ (for instance) concludes the proof.
\end{proof}

\section{Upper bounds}\label{sec:upper}

In this section we prove the ``density'' part of Theorem~\ref{thm:main}.

\begin{proposition}\label{prop:upper-bound-combined}
For each positive integer $d$, there is a constant $C=C_d>0$ such that the following holds.  Let $p>0$, sample a $p$-random subset $A \subseteq \mathbb{Z}_{\geq 0}^d$, and set $S^*:=\FS(A)$.  Then with high probability, we have
$$\mathbb{Z}_{\geq 0}^d \setminus\mathcal{R}_d(p,C p^{-1}(\log p^{-1})^{d+1}) \subseteq S^*.$$
\end{proposition}

%The main work is showing that $S^*$ contains the points of $\mathcal{R}_d(p,C p^{-1}(\log p^{-1})^{d+1})$ where all coordinates are at scale at least $\log p^{-1}$; at the end we will use induction on dimension to handle the points closer to the coordinate planes.  
The basic mechanism is illustrated most simply in $1$ dimension:  Suppose we want to show that a particular $x \in \mathbb{Z}_{\geq 0}$ is contained in $S^*$ with very high probability.  We initially reveal the elements of $A$ below $x/2$ and afterwards reveal the elements above $x/2$.  The initial reveal determines $S^* \cap [0,x/2)$.  For each $y \in S^* \cap [0,x/2)$, the ``complementary'' element $x-y$ is in $A$ with probability $p$, and these events are independent for different $y$'s.  In particular, if we know that $S^* \cap [0,x/2)$ is substantially larger than $p^{-1}$, then with very high probability $S^*$ will contain $x$.

In order to obtain $p^{-1}$ small elements of $S^*$, we certainly need to look at subset sums involving at least the first $\log_2 p^{-1}$ elements of $A$; this many elements of $A$ can be found at scale $p^{-1} \log p^{-1}$.  A typical subset sum of these early elements of $A$ is at scale $p^{-1} (\log p^{-1})^2$, so the argument from the previous paragraph kicks in once $x$ is at scale $p^{-1} (\log p^{-1})^2$.  This heuristic provides a conceptual explanation for the form of the threshold in Theorem~\ref{thm:1-dim}.

In dimension $d$, we can find $\log p^{-1}$ elements of $A$ in any box (with one corner at the origin) of volume at scale $p^{-1} \log p^{-1}$; a typical subset sum dilates \emph{each} coordinate by a further factor of $\log p^{-1}$ and hence has product of coordinates at scale $p^{-1}(\log p^{-1})^{d+1}$.  Thus we can expect $S^*$ to contain a point $x=(x_1, \ldots, x_d)$ once $x_1 \cdots x_d \gg p^{-1}(\log p^{-1})^{d+1}$ and each $x_i \gg \log p^{-1}$.  A different approach, using induction on dimension, is necessary to handle points $x$ with some ``small'' coordinates $x_i \ll \log p^{-1}$.

\subsection{Subset sums in finite abelian groups}
Our first task is showing that if $G$ is a finite abelian group of order $|G|=q$ and $A \subseteq G$ is a random subset of size much larger than $\log q$, then with very high probability $\FS(A)=G$.  We use a standard Poissonization argument to reduce to a situation that is treated in \cite{KMS} using the second-moment method.  Alternatively, one could appeal to classical (and slightly stronger) work of Erd\H{o}s and R\'enyi~\cite{ER} or of Alon and Roichman~\cite{AR}, which use variants of the same second-moment calculation; we include the reduction to \cite{KMS} (where the argument is short and elementary) for the reader's convenience.  For recent work on related questions, see~\cite{MT} and the references therein.

\begin{lemma}\label{lem:finite-ab-gp}
Let $q,L \in \mathbb{Z}_{>0}$ with $\log_2 q \leq L \leq q^{1/2}$.  Let $G$ be an abelian group of order $q$.  If $A \subseteq G$ is a uniformly random subset of size $L$, then $\FS(A)=G$ with probability at least $$1-q^2 \cdot 2^{1-L}.$$
\end{lemma}

\begin{proof}
Sample $g_1, \ldots, g_L \in G$ independently and uniformly at random (with replacement).  Let $X$ denote the event that $g_1, \ldots, g_L$ are all distinct, and let $Y$ denote the event that
\[\left\{\varepsilon_1g_1+\cdots+\varepsilon_L g_L:\varepsilon_1,\ldots,\varepsilon_L\in\{0,1\}\right\}=G.\]
The second-moment argument from \cite[Proposition 3.1]{KMS}\footnote{The lemma there is stated only for finite cyclic groups but, as remarked afterwards, goes through verbatim for general finite abelian groups.} tells us that
$$\mathbb{P}[Y] \geq 1 -q^2 \cdot 2^{-L}.$$
At the same time, we have
\begin{align*}
\mathbb{P}[Y] &=\mathbb{P}[X] \cdot \mathbb{P}[Y|X \text{ occurs}]+(1-\mathbb{P}[X]) \cdot \mathbb{P}[Y|X \text{ does not occur}]\\
 &\leq \mathbb{P}[X] \cdot \mathbb{P}[Y|X \text{ occurs}]+(1-\mathbb{P}[X]).
\end{align*}
Combining these two inequalities and rearranging, we find that
$$\mathbb{P}[Y|X \text{ occurs}] \geq 1-\frac{q^2 \cdot 2^{-L}}{\mathbb{P}[X]}.$$
We can estimate
$$\mathbb{P}[X]=\prod_{i=0}^{L-1} (1-i/q) \geq 1-\sum_{i=0}^{L-1} i/q=1-\frac{L(L-1)}{2q} \geq 1/2$$
(in the last inequality using the assumption $L \leq q^{1/2}$).  Thus in fact
$$\mathbb{P}[Y|X \text{ occurs}] \geq 1-q^2 \cdot 2^{1-L}.$$
Conditioning on $X$ occurring makes $\{g_1, \ldots, g_L\}$ into a uniformly random $L$-element subset of $G$, so the probability described in the lemma statement is precisely $\mathbb{P}[Y|X \text{ occurs}]$.
\end{proof}

We record a consequence which bears more directly on the setting of Proposition~\ref{prop:upper-bound-combined}.  The following lemma produces a ``dense spot'' in $\FS(A)$ which will power the ``complementary elements'' argument described in the proof overview.

\begin{lemma}\label{lem:ab-gp-lift}
For every real $\gamma>0$, there is a constant $C=C(\gamma)>0$ such that the following holds.  Let $m$ be a positive integer and let $p>0$, and suppose that $Y_1, \ldots, Y_m \in \mathbb{Z}_{>0}$ satisfy $$Y:=Y_1 \cdots Y_m \geq C p^{-1}(\log p^{-1}).$$
Sample a $p$-random subset $A \subseteq [1,Y_1] \times \cdots \times [1,Y_m]$.  Then with probability $1-O_{\gamma}(p^{\gamma})$, we have
$$|\FS(A) \cap ([0,pYY_1] \times \cdots \times [0,pYY_m])| \geq Y.$$
\end{lemma}
%\VB{Add Picture Rectangles}

\begin{proof}
Consider the finite abelian group
$$G:=(\mathbb{Z}/Y_1 \mathbb{Z}) \times \cdots \times (\mathbb{Z}/Y_m \mathbb{Z})$$
of order $|G|=Y_1 \cdots Y_m=Y$.  Notice that the ``host box'' $[1,Y_1] \times \cdots \times [1,Y_m]$ for $A$ is a fundamental domain in $\mathbb{Z}^m$ for $G$ under the canonical projection map.  A Chernoff bound tells us that
$$\mathbb{P} \left[|A| \leq p|G|/2\right] \leq \exp(-p|G|/8)=\exp(-pY/8) \leq p^{C/8}.$$
Suppose we are in an outcome with $|A| \geq p|G|/2$.  Choose a uniformly random subset $A' \subseteq A$ of size $\lceil p|G|/2\rceil$; its projection of to $G$ is a uniformly random subset of the same size.  So Lemma~\ref{lem:finite-ab-gp} tells us that with probability at least
$$1-|G|^2 \cdot 2^{1-\lceil p|G|/2\rceil}\geq 1-p^{-\Omega(C)},$$
the set $\FS(A')$ covers all of the residues modulo $G$, and in particular $|\FS(A')| \geq |G|=Y$.  Since $$\FS(A') \subseteq  [0,|A'|Y_1] \times \cdots \times [0,|A'|Y_m] \subseteq [0,pYY_1] \times \cdots \times [0,pYY_m]$$
by design, we conclude by a union bound that
$$\mathbb{P}[|\FS(A) \cap ([0,pYY_1] \times \cdots \times [0,pYY_m])| \geq Y] \geq 1-p^{C/8}-p^{-\Omega(C)};$$
this quantity is $1-O_\gamma(p^{\gamma})$ if $C$ is sufficiently large in terms of $\gamma$.
\end{proof}

\subsection{The bulk}

Armed with the tools from the previous subsection, we show that with high probability $S^*$ contains all of the desired points away from the coordinate planes; this is the main component of the proof of Proposition~\ref{prop:upper-bound-combined}.

\begin{proposition}\label{prop:upper-bound-bulk}
For each positive integer $m$, there is a constant $C'=C'_m>0$ such that the following holds.  Let $p>0$, sample a $p$-random subset $A \subseteq \mathbb{Z}_{\geq 0}^m$, and set $S^*:=\FS(A)$.  Then with high probability, we have
$$\{(x_1, \ldots, x_m) \in \mathbb{Z}_{\geq 0}: x_1\cdots x_m \geq C' p^{-1}(\log p^{-1})^{m+1} ~\text{and} ~x_i \geq C' \log p^{-1} ~\forall i\} \subseteq S^*.$$
\end{proposition}

As in the proof of Proposition~\ref{prop:lower}, we cover the region of interest with several simpler sets that can be handled individually.  For a positive integer $m$ and parameters $Y_1, \ldots, Y_m>0$, define
$$\mathcal{W}(Y_1, \ldots, Y_m):=\left\{ (y_1, \ldots, y_d) \in \mathbb{Z}^m_{\geq 0}: y_i \geq Y_i ~\forall i \right\}.$$
It is easy to approximate the region from Proposition~\ref{prop:upper-bound-bulk} with regions of the form $\mathcal W(Y_1, \ldots, Y_m)$.

\begin{lemma}\label{lem:upper-dyadic}
Let $m$ be a positive integer, and let $p,\kappa,Z>0$ be real parameters.  Then
$$\{(x_1, \ldots, x_m) \in \mathbb{Z}_{\geq 0}: x_1\cdots x_m \geq Z ~\text{and} ~x_i \geq \kappa \log p^{-1} ~\forall i\} \subseteq \bigcup_{(e_1, \ldots, e_m) \in \mathcal{E}'} \mathcal{W}(2^{e_1}, \ldots, 2^{e_m}),$$
where we have defined the index set
$$\mathcal{E}':=\{(e_1, \ldots, e_m) \in \mathbb{Z}^m: \log_2(\kappa \log p^{-1})-1 \leq e_i \leq \log_2(Z) ~\forall i \quad \text{and} \quad  e_1+\cdots+e_d \geq \log_2(Z)-m\}.$$
\end{lemma}

\begin{proof}
We proceed much as in the proof of Lemma~\ref{lem:covering with hyperplanes}.  Suppose $(x_1, \ldots, x_m) \in \mathbb{Z}_{\geq 0}$ satisfies
$$x_1\cdots x_m \geq Z \quad \text{and} \quad x_i \geq \kappa \log p^{-1} ~\forall i.$$
Define the exponents
$$e_i:=\min\{\lfloor \log_2(x_i) \rfloor, \lfloor \log_2(Z)\rfloor\}$$
for $1 \leq i \leq m$.  It is immediate that $(x_1, \ldots, x_m) \in \mathcal{W}(2^{e_1}, \ldots, 2^{e_m})$, and from
$$e_1+\cdots+e_m \geq \min\{\log_2(Z)-1,\log_2(x_1\cdots x_m)-m\} \geq \log_2(Z)-m$$
we see that $(e_1, \ldots, e_m) \in \mathcal{E}'$.
\end{proof}

We now show that with very high probability, $S^*=\FS(A)$ contains $\mathcal{W}(Y_1, \ldots, Y_m)$ whenever $Y_1 \cdots Y_m$ is at least around $p^{-1}(\log p^{-1})^{m+1}$.  See \cref{fig:schematic-prooflem3.6} for a schematic of the various pieces of the argument.

\begin{proposition}\label{prop:upper-single-box}
For each positive integer $m$ and real $\gamma>0$, there is a constant $C'=C'_m(\gamma)>0$ such that the following holds.  Let $p>0$, and suppose that $Y_1, \ldots, Y_m \geq C' \log p^{-1}$ satisfy $$Y_1 \cdots Y_m \geq C' p^{-1}(\log p^{-1})^{m+1}.$$
Sample a $p$-random subset $A \subseteq \mathbb{Z}_{\geq 0}^m$, and set $S^*:=\FS(A)$.  Then with probability $1-O_{m,\gamma}(p^{\gamma})$, we have
$$\mathcal{W}(Y_1, \ldots, Y_m) \subseteq S^*.$$
\end{proposition}

%\VB{Add Picture of the $\mathcal{W}$.}

\begin{proof}
Let $C'=C'_m(\gamma)>1$ be a large constant to be determined later.  We may assume that $\gamma \geq 1$ and that $p$ is sufficiently small in terms of $m,\gamma$.  By possibly decreasing some of the $Y_i$'s, we may assume that their product is (for instance) $$Y:=Y_1 \cdots Y_m \leq 2C' p^{-1}(\log p^{-1})^{m+1}.$$

We begin by deploying Lemma~\ref{lem:ab-gp-lift}.  Set
$$Y'_i:=\left\lceil \frac{Y_i}{D\log p^{-1}} \right\rceil$$
for $1 \leq i \leq m$, where $D=D_m(\gamma)$ is a constant satisfying $0<D\leq C'$ to be determined later.  Then
$$Y':=Y'_1 \cdots Y'_m \asymp_m \frac{Y}{D^m (\log p^{-1})^m} \asymp C'D^{-m} p^{-1} \log p^{-1}.$$
If $C',D$ are chosen so that $C' D^{-m}$ is sufficiently large relative to the constant $C(\gamma)$ from Lemma~\ref{lem:ab-gp-lift}, then this lemma implies that with probability $1-O_{\gamma}(p^\gamma)$, the set
$$A_0:=A \cap [1,Y'_1] \times \cdots \times [1,Y'_m]$$
satisfies
$$|\FS(A_0) \cap ([0,pY'Y'_1] \times \cdots \times [0,pY'Y'_m])| \geq Y'.$$
Suppose we are in an outcome for $A_0$ where this is the case.  Notice that
$$pY'Y'_i \asymp_m p \cdot \frac{C'p^{-1} \log p^{-1}}{D^m} \cdot \frac{Y_i}{D \log p^{-1}}=C' D^{-(m+1)} Y_i$$
is smaller than $Y_i/2$ (say) as long as $C',D$ are chosen to make $C'D^{-(m+1)}$ sufficiently small.

Next, we handle the points of $\mathcal{W}(Y_1, \ldots, Y_m)$ with all coordinates fairly small.  Let
$$\mathcal{W}'(Y_1, \ldots, Y_m):= \mathcal{W}(Y_1, \ldots, Y_m) \cap [1,\lfloor p^{-2} \rfloor]^m$$
(the particular choice of $p^{-2}$ is not important).  For each $x \in \mathcal{W}'(Y_1, \ldots, Y_m)$ consider the points $x-z$ with
$$z \in \FS(A_0) \cap ([0,pY'Y'_1] \times \cdots \times [0,pY'Y'_m]).$$
The number of such points $x-z$ is at least $Y'$, and they all lie outside of the region from which $A_0$ was sampled.  Thus, when we reveal $A_1:=(A \cap [1,\lfloor p^{-2}\rfloor ]^m) \setminus A_0$, with probability at least
$$1-(1-p)^{Y'} \geq 1-e^{-pY'} \geq 1-p^{\Omega_m(C' D^{-m})},$$
some such point $x-z$ lies in $A_1$ and thus
$$x=z+(x-z) \in \FS(A_0)+(A_1) \subseteq \FS(A_0 \sqcup A_1).$$
A union bound over $x \in \mathcal{W}'(Y_1, \ldots, Y_m)$ gives that $\mathcal{W}'(Y_1, \ldots, Y_m) \subseteq \FS(A_0 \sqcup A_1) \subseteq \FS(A)$ with probability at least
$$1-p^{-2m+\Omega_m(C' D^{-m})};$$
this probability is $1-O_{m,\gamma}(p^\gamma)$ as long as $C'D^{-m}$ is sufficiently large in terms of $m,\gamma$.  To summarize, we have shown that
$$\mathbb{P}[\mathcal{W}'(Y_1, \ldots, Y_m) \subseteq \FS(A)] \geq 1-O_{m,\gamma}(p^\gamma)$$
as long as the constants $C'\geq D>0$ are chosen such that $C'D^{-m}$ is sufficiently large and $C'D^{-(m+1)}$ is sufficiently small (in terms of $m,\gamma$), which is clearly possible.

Finally, we handle the points of $$\mathcal{W}(Y_1, \ldots, Y_m) \setminus \mathcal{W}'(Y_1, \ldots, Y_m).$$  Let $u:=|\mathcal{W}'(Y_1, \ldots, Y_m)|$.  Fix an enumeration $w_1, w_2, \ldots$ of the points of $\mathcal{W}(Y_1, \ldots, Y_m)$ that begins with the elements of $\mathcal{W}'(Y_1, \ldots, Y_m)$ and has the property that $i \geq j$ whenever $w_i$ is coordinate-wise greater than or equal to $w_j$. Let $E_i$ denote the event that $w_i \in \FS(A)$.  These events are non-negatively correlated, so for any subsets $J_i \subseteq [1,i-1]$, we have
\begin{align*}
\mathbb{P}[E_i \text{ fails for some $i>u$}~ |~ E_1, \ldots, E_{u} \text{ all occur}] &=\sum_{i=u+1}^\infty \mathbb{P}[E_i \text{ fails}~ |~ E_1, \ldots, E_{i-1} \text{ all occur} ]\\
 &\leq \sum_{i=u+1}^\infty \mathbb{P}[E_i \text{ fails}~ |~ \text{$E_j$ occurs $\forall j \in J_i$}].
\end{align*}
For each index $i>u$, let $k_i$ denote the largest coordinate of $w_i$ (choosing arbitrarily in case of a tie), and let $J_i$ be the set of indices $j<i$ such that $w_j$ is coordinate-wise less than or equal to $w_i$ and the $k_i$-th coordinate of $w_j$ is less than half as large as the $k_i$-th coordinate of $w_i$.  The definition of $\mathcal{W}'(Y_1, \ldots, Y_m)$ ensures that $$|J_i| \geq \|w_i\|_\infty/3$$
for all $i>u$, where $\|\cdot \|_\infty$ denotes the $\ell^\infty$-norm (largest coordinate).  By considering which of the points $w_i-w_j$ lie in $A$, we conclude as above that
$$\mathbb{P}[E_i \text{ fails } | \text{$E_j$ occurs $\forall j \in J_i$}] \leq (1-p)^{|J_i|} \leq e^{-p\|w_i\|_\infty/3}$$
for each $i>u$.
We have $\|w_i\|_\infty \geq p^{-2}$ for all $i>u$.  For each $n \geq p^{-2}$, the number of indices $i>u$ with $\|w_i\|_\infty=n$ is $O_m(n^{m-1})$.  Thus we find that
$$\mathbb{P}[E_i \text{ fails for some $i>u$} ~| ~ \mathcal{W}'(Y_1, \ldots, Y_m) \subseteq \FS(A)] \ll_m \sum_{n \geq p^{-2}} n^{m-1} \cdot e^{-pn/3} \ll_m e^{-p^{-1}/4} \ll_\gamma p^{\gamma},$$
so
$$\mathbb{P}[\mathcal{W}(Y_1, \ldots, Y_m) \subseteq \FS(A) ~|~ \mathcal{W}'(Y_1, \ldots, Y_m) \subseteq \FS(A)] \geq 1-O_{m,\gamma}(p^\gamma).$$
Combining this with the outcome of the previous paragraph yields the proposition.
\end{proof}

\begin{figure}[ht]
\centering

\begin{tikzpicture}[scale=6]

% --- Define coordinates (spread out a bit more proportionally) ---
\def\YpOne{0.09}
\def\pYYpOne{0.27}
\def\YOne{0.6}
\def\pinv{1}

\def\YpTwo{0.06}
\def\pYYpTwo{0.18}
\def\YTwo{0.4}

% --- Yellow region: outside [0,p^{-2}]^2 ---
\fill[yellow!40] (\YOne,\YTwo) rectangle (2,1.5);
\fill[white] (\YOne,\YTwo) rectangle (\pinv,\pinv);

% --- Orange region ---
\fill[orange!50] (\YOne,\YTwo) rectangle (\pinv,\pinv);
%\fill[white] (0,0) rectangle (\YOne,\YTwo);

% --- Red region ---
\fill[red!40] (0,0) rectangle (\pYYpOne,\pYYpTwo);
\fill[white] (0,0) rectangle (\YpOne,\YpTwo);

% --- Blue region ---
\fill[blue!40] (0,0) rectangle (\YpOne,\YpTwo);

% --- Draw boundaries (thin black lines) ---
\draw[thin] (0,0) rectangle (\YpOne,\YpTwo);
\draw[thin] (0,0) rectangle (\pYYpOne,\pYYpTwo);
%\draw[thin] (0,0) rectangle (\YOne,\YTwo);
%\draw[thin] (0,0) rectangle (\pinv,\pinv);
\draw[thin] (\YOne,\YTwo) rectangle (\pinv,\pinv);
\draw[<->] (\YOne,1.5) -- (\YOne,\YTwo) -- (2,\YTwo);
\draw[dashed] (\YOne,0) -- (\YOne, \YTwo) -- (0,\YTwo);
\draw[dashed] (\pinv,0) -- (\pinv, \YTwo);
\draw[dashed] (\YOne, \pinv) -- (0,\pinv);

% --- Axes ---
\draw[->] (0,0) -- (2,0) node[right] {$x_1$};
\draw[->] (0,0) -- (0,1.5) node[above] {$x_2$};

% --- Tick marks (x-axis) ---
\draw (\YpOne,0) -- (\YpOne,-0.015) node[below] {$Y'_1$};
\draw (\pYYpOne,0) -- (\pYYpOne,-0.015) node[below] {$pYY'_1$};
\draw (\YOne,0) -- (\YOne,-0.015) node[below] {$Y_1$};
\draw (\pinv,0) -- (\pinv,-0.015) node[below] {$p^{-2}$};

% --- Tick marks (y-axis) ---
\draw (0,\YpTwo) -- (-0.015,\YpTwo) node[left] {$Y'_2$};
\draw (0,\pYYpTwo) -- (-0.015,\pYYpTwo) node[left] {$pYY'_2$};
\draw (0,\YTwo) -- (-0.015,\YTwo) node[left] {$Y_2$};
\draw (0,\pinv) -- (-0.015,\pinv) node[left] {$p^{-2}$};

% --- Region labels (repositioned for clarity) ---
\node at (0.03,0.03) {$A_0$};

\node at (0.15,0.1) {$\FS(A_0)$};

\node at (0.8,0.72) {$\mathcal W'(Y_1,Y_2)$};

\node at (1.3,1.15) {$\mathcal W(Y_1,Y_2)\setminus \mathcal W'(Y_1,Y_2)$};

\end{tikzpicture}

\caption{The regions involved in the proof of \cref{prop:upper-single-box} for $m=2$.  The ``seed'' set $A_0$ lies in the small blue rectangle.  Its set of subset sums $\FS(A_0)$ contains many elements of the larger red rectangle.  Next, $\FS(A)$ contains all of the points of the orange rectangle $\mathcal W'(Y_1,Y_2)$.  Finally, $\FS(A)$ contains all of the points of the (unbounded) yellow region $\mathcal W(Y_1,Y_2)\setminus \mathcal W'(Y_1,Y_2)$.}
\label{fig:schematic-prooflem3.6}

\end{figure}
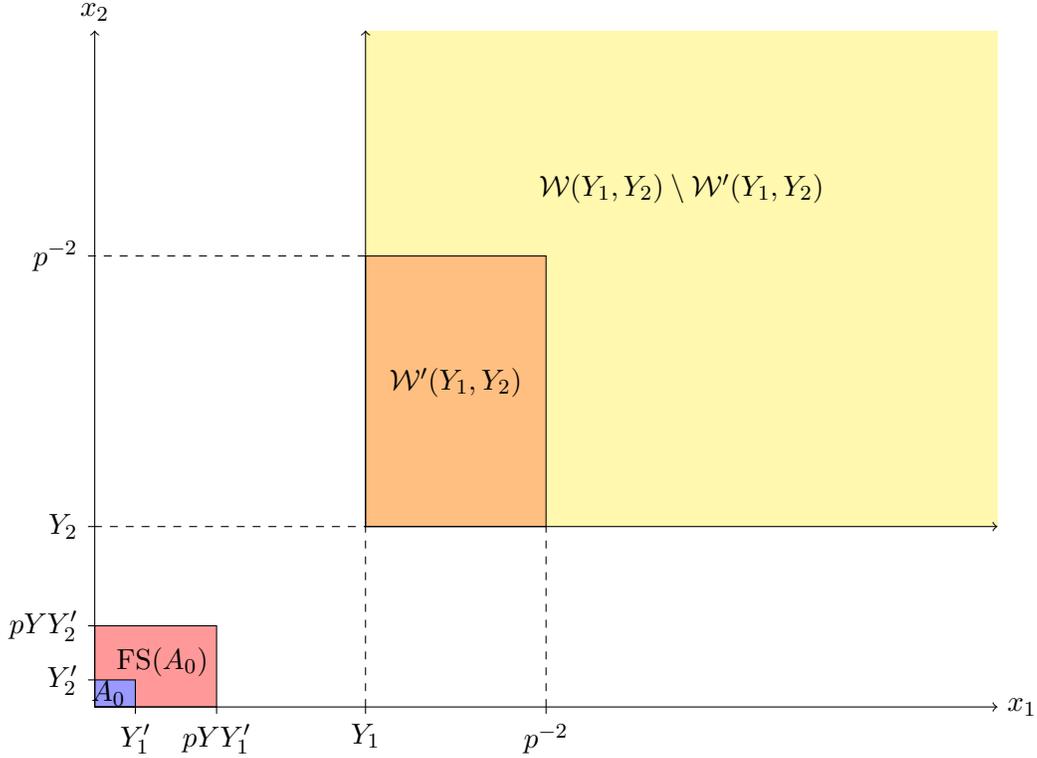

%\VB{Need to add also $A_1$? And perhaps modify the proportion, as $\frac{Y_1}{2}$ does nor seem to be really half of $Y_1$ (and same on the vertical axis)}

It is now straightforward to deduce Proposition~\ref{prop:upper-bound-bulk}.  We remark that we do not need the full strength of the arbitrary polynomial savings in Proposition~\ref{prop:upper-single-box}; sufficiently good polylogarithmic savings would serve equally well.

\begin{proof}[Proof of Proposition~\ref{prop:upper-bound-bulk}]
Let $C'$ be $2^m$ times the constant $C'_m(1)$ from Proposition~\ref{prop:upper-single-box}.  Let $\mathcal{E}'$ be the index set produced by Lemma~\ref{lem:upper-dyadic} with the parameters $\kappa:=C'$ and $Z:=C' p^{-1} \log p^{-1}$, and note that
$$|\mathcal{E}'| \ll_m (\log Z)^m \ll_m (\log p^{-1})^m.$$
For each $(e_1, \ldots, e_m) \in \mathcal{E}'$, apply Proposition~\ref{prop:upper-single-box} with $\gamma:=1$ (for instance) and $$(Y_1, \ldots, Y_m):=(2^{e_1}, \ldots, 2^{e_m})$$ to deduce that $\mathcal{W}(2^{e_1}, \ldots, 2^{e_m}) \subseteq \FS(A)$ with probability $1-O_m(p)$.  A union bound over $\mathcal{E}'$ tells us that 
$$\bigcup_{(e_1, \ldots, e_m) \in \mathcal{E}'} \mathcal{W}(2^{e_1}, \ldots, 2^{e_m}) \subseteq \FS(A)$$
with probability at least
$$1-O_m(p(\log p^{-1})^m)=1-o_m(1);$$
this yields the desired conclusion due to Lemma~\ref{lem:upper-dyadic}.
\end{proof}

\subsection{Alternative perspective via syndeticity}\label{sec:syndetic}
We digress to describe an alternative perspective on the arguments from the preceding subsection; the present subsection is not logically necessary for the rest of the paper and can be safely skipped.

In the $1$-dimensional instance of Proposition~\ref{prop:upper-bound-combined}, we eventually aim to show that with high probability $S^*=\FS(A)$ contains all positive integers of size at least $Cp^{-1}(\log p^{-1})^2$.  We begin with the more modest goal of showing that with high probability $S^*$ is syndetic with syndeticity constant at most around $p^{-1}\log p^{-1}$.  The argument is powered by the following simple lemma.

\begin{lemma}\label{lem:syndetic}
Let $R \in \mathbb{Z}_{>0}$.  If $a_1, a_2, \ldots$ is an increasing sequence of positive integers such that $a_1\leq R$ and $a_i\leq a_1+\cdots+a_{i-1}+R$ for all $i>1$, then $\FS(\{a_1, a_2, \ldots\})$ intersects every interval of $R$ consecutive nonnegative integers.
\end{lemma}

For the proof (whose details we omit), one shows inductively that $\FS(\{a_1, \ldots, a_i\})$ intersects every interval of $R$ consecutive integers in $[0,a_1+\cdots+a_i+R-1]$.

One can readily check that for $R$ of size at least around $p^{-1} \log p^{-1}$, with high probability the elements of a $p$-random subset $A \subseteq \mathbb{Z}_{\geq 0}$ satisfy the hypotheses of Lemma~\ref{lem:syndetic}.  It will be convenient to obtain the same conclusion for a $p$-random subset $A \subseteq [R/4,R/2] \cup [2R,\infty)$, as in the following lemma (whose proof we again omit).

\begin{lemma}\label{lem:syndetic-random}
Let $p>0$, and suppose that $R=R(p) \geq 100p^{-1} \log p^{-1}$ is a multiple of $4$.  Sample a $p$-random subset $$A \subseteq [R/4,R/2] \cup [2R,\infty).$$  Then with high probability $\FS(A)$ intersects every interval of $R$ consecutive nonnegative integers.
\end{lemma}

We now sketch a simple deduction of the $1$-dimensional instance of Proposition~\ref{prop:upper-bound-bulk}.  Set $R:=(C/2) p^{-1} (\log p^{-1})^2$ (which we assume to be a multiple of $4$), and partition $A=A_1 \sqcup A_2$, where
$$A_1:=A \cap ([R/4,R/2] \cup [2R,\infty)) \quad \text{and} \quad A_2:=A \cap ([0,R/4-1] \cup [R/2+1,2R-1]).$$
Lemma~\ref{lem:syndetic-random} tells us that with high probability $\FS(A_1) \subseteq \mathbb{Z}_{\geq 0}$ is $R$-syndetic.  Arguing as in the second and third paragraphs of the proof of Proposition~\ref{prop:upper-single-box} (but more simply since our box is now an interval), we find that with high probability $[R,2R-1] \subseteq \FS(A_2)$, as long as $C$ is sufficiently large.  Then $\FS(A)=\FS(A_1)+\FS(A_2)$ contains all integers greater than or equal to $R$, as desired.

The higher-dimensional instances of Proposition~\ref{prop:upper-bound-bulk} go similarly.  Roughly speaking, one splits $$A=A_1 \sqcup \cdots \sqcup A_{m} \sqcup A_*,$$ where each $A_i$ is the part of $A$ lying on the $i$-th coordinate axis, and $A_{*}$ is the rest of $A$.  Applying Lemma~\ref{lem:syndetic-random} to all of the $A_i$'s individually and summing gives that with high probability $$\FS(A \setminus A_*)=\FS(A_1)+\cdots+\FS(A_m) \subseteq \mathbb{Z}_{\geq 0}^m$$ is $[1,R]^m$-syndetic.   Let $X$ denote the set appearing in the statement of Proposition~\ref{prop:upper-bound-bulk}, and let $X':=\{x \in X: x-[1,R]^m \not\subseteq X\}$ be the part of $X$ ``within distance $R$ of the boundary''.  Lemma~\ref{lem:upper-dyadic} and the argument from the beginning of the proof of Proposition~\ref{prop:upper-single-box} together show that with high probability $X' \subseteq \FS(A_*)$.  Further adding $\FS(A \setminus A_*)$ covers the rest of $X$.

\subsection{Using lower-dimensional information}
We prove Proposition~\ref{prop:upper-bound-combined} by stitching together the output of Proposition~\ref{prop:upper-bound-bulk} for various dimensions $1 \leq m \leq d$.

\begin{proof}[Proof of Proposition~\ref{prop:upper-bound-combined}]
Let $C,D$ be constants (depending only on $d$) to be specified later.  For brevity, write $\mathcal{U}:=\mathbb{Z}_{\geq 0}^d \setminus\mathcal{R}_d(p,C p^{-1}(\log p^{-1})^{d+1})$.  We partition $\mathcal{U}$ according to how many coordinates are small, as follows.  For each nonempty subset $I \subseteq [d]$, let $\mathcal{U}_I$ be the set of points $(x_1, \ldots, x_d) \in \mathcal{U}$ such that
$$\{i \in [d]: x_i \geq D\log p^{-1}\}=I.$$
Notice that $\mathcal{U}$ is the disjoint union of the sets $\mathcal{U}_I$, for $\emptyset \neq I \subseteq [d]$, since (as long as $p$ is sufficiently small relative to $C,D$) each point of $\mathcal{U}$ has at least one coordinate greater than or equal to $D\log p^{-1}$.  We will show that with high probability $S^*$ contains each $\mathcal{U}_I$, whence the proposition follows from a union bound over $I$.

Fix some nonempty $I \subseteq [d]$.  By permuting the coordinates, we may assume that $I=[m]$ for some $1 \leq m \leq d$.  Let $(x_1, \ldots, x_d) \in \mathcal{U}_I$.  For each index $i \leq m$ we have
$$x_i+\log p^{-1} \leq (1+1/D)x_i,$$
and for each index $i>m$ we have
$$x_i+\log p^{-1}< D(1+1/D) \log p^{-1}.$$
Plugging these estimates into the ``main'' constraint of $\mathcal{U}$ gives
\begin{align*}
C p^{-1} (\log p^{-1})^{d+1} & \leq \prod_{i \leq m}(x_i+\log p^{-1}) \cdot \prod_{i >m}(x_i+\log p^{-1})\\
 &\leq \prod_{i \leq m}(1+1/D)x_i \cdot \prod_{i >m} D(1+1/D) \log p^{-1}\\
 &=D^{d-m}(1+1/D)^d (\log p^{-1})^{d-|I|} \prod_{i \leq m}x_i.
\end{align*}
Taking $D$ sufficiently large ensures that $(1+1/D)^d \leq 2$ (say), so we see that $\mathcal{U}_I$ is contained in the nicer set
$$\mathcal{U}'_I:=\left\{
\begin{aligned}
(x_1,\ldots,x_m)\in\mathbb{Z}_{\geq 0}^m:  & ~ x_1\cdots x_m \geq (C D^{m-d}/2)p^{-1}(\log p^{-1})^{m+1}\\
 &\qquad \quad \text{and} ~x_i \geq D\log p^{-1} ~ \forall i  
\end{aligned}
\right\} \times [0,D\log p^{-1})^{d-m};$$
thus it suffices to show that $\FS(A)$ contains $\mathcal{U}'_I$ with high probability.  %Notice that the first multiplicand in $\mathcal{U}'_I$ closely resembles the $m$-dimensional set from Proposition~\ref{prop:upper-bound-bulk}.

Consider the part of $A$ in the coordinate plane where the last $d-m$ coordinates vanish, viz.
$$A_I:=A \cap (\mathbb{Z}_{\geq 0}^m \times \{0\}^{d-m}).$$
Applying Proposition~\ref{prop:upper-bound-bulk} in $m$ dimensions, we find that with high probability $\FS(A_I)$ contains
\begin{equation}\label{eq:single-slice}
\{(x_1, \ldots, x_m) \in \mathbb{Z}_{\geq 0}^m: x_1\cdots x_m \geq C'_{m}p^{-1}(\log p^{-1})^{m+1} ~\text{and} ~x_i \geq C'_{m} \log p^{-1} ~\forall i\} \times \{0\}^{d-m},
\end{equation}
where $C'_{m}$ is the constant from that proposition; of course then $\FS(A) \supseteq \FS(A_I)$ also contains \eqref{eq:single-slice}.  Henceforth assume that we are in such an outcome.  Now that we have handled the ``slice'' of $\mathcal{U}'_I$ where the second component is $0$, we will use it to obtain information about the remaining slices.

In anticipation of a dyadic decomposition, define the set of scales
$$\mathcal{E}:=\{(e_1, \ldots, e_m) \in \mathbb{Z}_{\geq 0}^m: e_1+\cdots+e_m=\lceil \log_2 (Ep^{-1} (\log p^{-1})^{m+1}) \rceil \},$$
where $E:=C D^{m-d}/10$ is another constant (whose precise value is not too important).  We henceforth omit ceiling functions for readability.  The number of ways to choose some $e=(e_1, \ldots, e_m) \in \mathcal{E}$ and some nonzero $z=(z_{m+1}, \ldots, z_{d}) \in [0,D\log p^{-1})^{d-m}$ is (crudely)
$$|\mathcal{E}| \cdot ((D\log p^{-1}+1)^{d-m}-1) \leq 2 D^{d-m} (\log p^{-1})^{d-1}.$$
For each such choice of $(e,z)$, consider the box
$$[1,2^{e_1}] \times \cdots \times [1,2^{e_m}] \times \{(z_{m+1}, \ldots, z_{d})\}.$$
This box has size $2^{e_1+\cdots+e_m}\geq Ep^{-1} (\log p^{-1})^{m+1}$, so with probability at least
$$1-(1-p)^{Ep^{-1} (\log p^{-1})^{m+1}} \geq 1-e^{-E(\log p^{-1})^{m+1}}$$
it contains some element $a_{(e,z)}$ of $A$.  A union bound ensures that with probability at least
$$1-2D^{d-m} (\log p^{-1})^{d-1} \cdot e^{-E(\log p^{-1})^{m+1}}=1-o_d(1),$$
this is the case for all pairs $(e,z)$ (note that the elements $a_{(e,z)}$ may not be distinct).  Assume that we are in such an outcome; it remains to show that then $\mathcal{U}'_I \subseteq S^*$.

Consider a point $(y_1, \ldots, y_m, z_{m+1}, \ldots, z_d) \in \mathcal{U}'_I$, and set $z:=(z_{m+1}, \ldots, z_d)$.  Due to the choice of $E$, there is (with room to spare) some $e \in \mathcal{E}$ such that $y_i \geq 2^{e_i+1}$ for all $1 \leq i \leq m$.  Take the element $a_{(e,z)}$ of $A$ from the previous paragraph.  Then $S^*$ contains $a_{(e,z)}+\FS(A_I)$,
which in turn (due to \eqref{eq:single-slice}) contains all of the points $(x_1, \ldots, x_m, z_{m+1}, \ldots, z_d) \in \mathbb{Z}_{\geq 0}^d$ with
$$(x_1-2^{e_1})\cdots (x_m-2^{e_m}) \geq C'_{m}p^{-1}(\log p^{-1})^{m+1} \quad \text{and} \quad x_i-2^{e_i} \geq C'_{m} \log p^{-1} ~\forall i.$$
Let us verify that our chosen point $(y_1, \ldots, y_m, z_{m+1}, \ldots, z_d)$ satisfies these two conditions.  For the second condition, the inequality $y_i-2^{e_i} \geq y_i/2$ and the definition of $\mathcal{U}'_I$ give
$$y_i-2^{e_i} \geq y_i/2 \geq (D/2)\log p^{-1};$$ this is acceptable as long as $D \geq 2C'_m$.  For the first condition, we have
$$(y_1-2^{e_1})\cdots (y_m-2^{e_m}) \geq 2^{-m}y_1 \cdots y_m \geq (2^{-m-1}C D^{m-d})p^{-1}(\log p^{-1})^{m+1};$$
this is again acceptable as long as $C$ is sufficiently large relative to $D,C'_m$.

This completes the proof that for such a choice of constants, with high probability $S^*$ contains each $\mathcal{U}'_I$ and hence contains $\mathcal{U}$.
\end{proof}

\section{Assembling the pieces}\label{sec:assembling}
The main theorems follow immediately from the results of the previous two sections and the deterministic inclusion $S^*=\FS(A) \subseteq \langle A \rangle=S$.  Indeed, Propositions~\ref{prop:lower} and~\ref{prop:upper-bound-combined} provide constants $c=c_d, C=C_d>0$ such that with high probability
$$|S^* \cap \mathcal{R}_d(p,c p^{-1}(\log p^{-1})^{d+1})| \leq |S \cap \mathcal{R}_d(p,c p^{-1}(\log p^{-1})^{d+1})|=o(|\mathcal{R}_d(p,c p^{-1}(\log p^{-1})^{d+1})|)$$
and
$$\mathbb{Z}_{\geq 0}^d \setminus \mathcal{R}_d(p,C p^{-1}(\log p^{-1})^{d+1}) \subseteq S^* \subseteq S;$$
this establishes the first two statements of each of Theorems~\ref{thm:main} and~\ref{thm:subset-sums} (on taking $C$ there to be the maximum of $1/c$ and the constant $C$ here).  The ``in particular'' statements follow from the estimate of $|\mathcal{R}_d(p,Cp^{-1}(\log p^{-1})^{d+1})|$ in Lemma~\ref{lem:appendix}.

\section{Concluding remarks}\label{sec:concluding}

We conclude with some remarks and open problems.

\subsection{Strengthenings of our main results}

An examination of the proof of Proposition~\ref{prop:upper-single-box} reveals that the only summands ever used are the elements of $A$ that lie on the coordinate axes or have all coordinates at most $p^{-2}$.  By running the argument more carefully, one can get by with only the elements of $A$ that lie on the coordinate axes or in the box $[0,2Y_1] \times \cdots \times [1,2Y_m]$.  Carrying this observation through the rest of the proof of Proposition~\ref{prop:upper-bound-combined}, we find that we can obtain the same conclusion for the restriction of $A$ to the union of the coordinate axes and $\mathcal{R}_d(p, \widetilde Cp^{-1}(\log p^{-1})^{d+1})$, where $\widetilde C>0$ is some constant slightly larger than $C$.  In other words, when $d>1$, the inclusion $\mathcal{R}_d(p, Cp^{-1}(\log p^{-1})^{d+1}) \subseteq \FS(A)$ can be explained using only the part of $A$ lying in a small region, which is a priori surprising.

The main result of \cite{KMS} (see Theorem~\ref{thm:1-dim} above) established the asymptotic values of the genus and Frobenius number not only with high probability but also in expectation.  Obtaining ``in expectation'' analogues of our Theorems~\ref{thm:main} and~\ref{thm:subset-sums} should be routine.  The work \cite{KMS} also determined the asymptotic with-high-probability and expected embedding dimension (minimum size of a generating set) of a random numerical semigroup.  By combining the arguments from \cite{KMS} with the work of the present paper, one can obtain analogous results in the multidimensional setting of Theorem~\ref{thm:main}.  More precisely, one can show that with high probability the minimal generating set of $\langle A \rangle$ is contained in $\mathcal{R}_d(p, Cp^{-1}(\log p^{-1})^{d+1})$ (for a suitable constant $C=C_d>0$) and contains a $\Omega_d(p^{-1})$-proportion of the elements of $\mathcal{R}_d(p, Cp^{-1}(\log p^{-1})^{d+1})$; and in particular the embedding dimension is $\asymp_d (\log p^{-1})^{2d}$.  The former statement can be obtained via the approach sketched in Section~\ref{sec:syndetic}, but with Lemma~\ref{lem:syndetic-random} replaced by the observation that with high probability $A$ contains some element $0<a \leq p^{-1} \log p^{-1}$ (say) and $\langle \{a\} \rangle \subseteq \langle A \rangle$ is already $\lfloor p^{-1} \log p^{-1} \rfloor$-syndetic.  We leave these matters to the interested reader.

\subsection{Similarity to other scaling limits}

We have presented Theorems~\ref{thm:main} and~\ref{thm:subset-sums} as identifying ``scaling limits'' for the gap sets of the random objects $S=\langle A \rangle$ and $S^*=\FS(A)$.  There is at least a topical similarity with other scaling limits, such as the limit shape for random integer partitions under the Plancherel measure (see \cite{LS,KV}).  It would be interesting to explore this resemblance further---both to find connections with scaling limits in probability theory and statistical physics, and to identify other random number-theoretic objects with describable limit shapes.

\subsection{Further questions}

Finally, we propose three ways in which it would be interesting to sharpen our results in future work.  The first question, of course, is whether the constant-factor gaps between lower and upper bounds in Theorems~\ref{thm:main} and~\ref{thm:subset-sums} are really necessary.  This problem already seems difficult in $1$ dimension.  In order to understand the higher-dimensional situation, one might need to devise a more ``sensitive'' scaling limit than our shifted hyperboloid.

A related question is whether the with-high-probability gap set shapes of $\langle A \rangle$ and $\FS(A)$ are macroscopically distinguishable.  Our main results imply that the scaling limits of these two gap sets agree up to dilation by a constant factor, and we ask if the scaling limits are in fact identical up to a $1+o(1)$ factor.  Again, even the $1$-dimensional case of this question seems quite challenging.

The third question is what one can say about \emph{pointwise} probabilities of particular elements being contained in $\langle A \rangle$ and $\FS(A)$.  For instance, are $\mathbb{P}[(x_1, \ldots, x_d) \in \langle A \rangle]$ and $\mathbb{P}[(x_1, \ldots, x_d) \in \FS(A)]$ roughly of the form $$1-\varphi_d\left(\frac{(x_1+\log p^{-1}) \cdots (x_d+\log p^{-1})}{p^{-1} (\log p^{-1})^{d+1}}\right)$$
for some quickly-decaying function $\varphi_d: \mathbb{R}_{>0} \to [0,1]$?  We do not even know if these probabilities enjoy monotonicity properties in the $x_i$'s when the $x_i$'s are large.  In $1$ dimension, for example, $$\mathbb{P}[4 \in \langle A \rangle]=1-(1-p)^3>1-(1-p)^2(1-p^2)=\mathbb{P}[5 \in \langle A \rangle]$$ shows that there can be non-monotonicity at least at small values.  More generally, one could study how sharply these probabilities jump near the interface of the scaling limit.

\appendix

\section{Volumes and lattice points under hyperboloids}\label{sec:appendix}

In this appendix we estimate the size of the shifted hyperboloid set $\mathcal{R}_d(p,Z)$, which plays a central role in this paper.  We start with the volume of the corresponding continuous region
$$\widetilde{R}_d(L,Z):=\{(x_1, \ldots, x_d) \in \mathbb{R}_{\geq 0}: (x_1+L) \cdots (x_d+L) \leq Z\}.$$

\begin{lemma}\label{lem:integral}
Let $d$ be a positive integer.  Let $L,Z>0$ be such that $Z \geq L^d$.  Then we have
$$\Vol_d(\widetilde{R}_d(L,Z))=\begin{cases}
Z-L, &\text{if }d=1;\\
Z\sum_{i=0}^{d-1} (-1)^i \cdot \frac{(\log Z-d\log L)^{d-1-i}}{(d-1-i)!}, &\text{if }d \geq 2.
\end{cases}$$
\end{lemma}

\begin{proof}
The calculation is easy for $d=1$, so we restrict our attention to the case $d \geq 2$.  Making several changes of variables, we obtain
\begin{align*}
\Vol_d(\widetilde{R}_d(L,Z)) &=\int_{\substack{\{(x_1+L) \cdots (x_d+L) \leq Z;\\ x_1, \ldots, x_d \geq 0\}}} 1 \,dx_1 \cdots dx_d\\
 &=\int_{\substack{\{x_1\cdots x_d \leq Z;\\ x_1, \ldots, x_d \geq L\}}} 1 \,dx_1 \cdots dx_d\\
 &=\int_{\substack{\{y_1+\cdots+y_d \leq \log Z;\\ y_1, \ldots, y_d \geq \log L\}}} e^{y_1+\cdots+y_d} \,dy_1 \cdots dy_d\\
 &=e^{d\log L}\int_{\substack{\{y_1+\cdots+y_d \leq \log Z-d\log L;\\ y_1, \ldots, y_d \geq 0\}}} e^{y_1+\cdots+y_d} \,dy_1 \cdots dy_d\\
 &=e^{d\log L} \int_{0}^{\log Z-d\log L}e^y \int_{\substack{\{y_1+\cdots+y_{d-1} \leq y;\\ y_1, \ldots, y_{d-1} \geq 0\}}} 1  \,dy_1 \cdots dy_{d-1} \,dy.
\end{align*}
(Here we performed linear shifts in the second and fourth lines; substituted $y_i:=\log x_i$ in the third line; and replaced $y_d$ by the new variable $y:=y_1+\cdots+y_d$ in the fifth line.)  The inner integral is the volume of the $(d-1)$-dimensional simplex whose vertices are the origin and $y$ times the canonical basis vectors; the volume of this simplex is $y^{d-1}/(d-1)!$.  Thus
\begin{align*}
\Vol_d(\widetilde{R}_d(L,Z)) &=e^{d\log L} \int_{0}^{\log Z-d\log L}\frac{e^y y^{d-1}}{(d-1)!} \,dy.
\end{align*}
The latter integral can be evaluated by repeated integration by parts, which yields
\begin{align*}
\Vol_d(\widetilde{R}_d(L,Z)) &=e^{d \log L} \sum_{i=0}^{d-1} (-1)^i \cdot \frac{e^{\log Z-d\log L}(\log Z-d\log L)^{d-1-i}}{(d-1-i)!}\\
 &=Z\sum_{i=0}^{d-1} (-1)^i \cdot \frac{(\log Z-d\log L)^{d-1-i}}{(d-1-i)!},
\end{align*}
as desired.
\end{proof}

We can now treat the discrete sets $\mathcal{R}_d(p,Z)$.

\begin{lemma}\label{lem:appendix}
Let $d$ be a positive integer, and let $C>0$ be a constant.  For $p>0$, we have
$$|\mathcal{R}_d(p,Cp^{-1}(\log p^{-1})^{d+1})|=(1+o_{d,C}(1)) \cdot \frac{C p^{-1}(\log p^{-1})^{2d}}{(d-1)!}.$$
\end{lemma}

\begin{proof}
Let us identify each point $$(y_1, \ldots, y_d) \in \mathcal{R}_d(p,Cp^{-1}(\log p^{-1})^{d+1})$$ with the unit cube $$\{(z_1, \ldots, z_d) \in \mathbb{R}^d: \lfloor z_i \rfloor=y_i ~\forall i\}.$$
The union of these cubes contains the continuous region $$\widetilde{R}_d(\log p^{-1},Cp^{-1}(\log p^{-1})^{d+1}),$$
so we have the lower bound
$$|\mathcal{R}_d(p,Cp^{-1}(\log p^{-1})^{d+1})| \geq \Vol(\widetilde{R}_d(\log p^{-1},Cp^{-1}(\log p^{-1})^{d+1})).$$
At the same time, the cubes are all disjoint and their union is contained in the slightly larger continuous region $$\widetilde{R}_d(\log p^{-1}-1,Cp^{-1}(\log p^{-1})^{d+1}),$$
so we have the upper bound
$$|\mathcal{R}_d(p,Cp^{-1}(\log p^{-1})^{d+1})| \leq \Vol(\widetilde{R}_d(\log p^{-1}-1,Cp^{-1}(\log p^{-1})^{d+1})).$$
The desired estimate now follows from Lemma~\ref{lem:integral} (note that the $i=0$ term dominates the sum for $d \geq 2$).
\end{proof}

\section*{Acknowledgments}
The first author is supported by the EPSRC Scholarship under grant EP/Z534870/1.  The second author was supported in part by a NSF Mathematical Sciences Postdoctoral Research Fellowship under grant DMS-2501336. We thank Ben Green and Santiago Morales for helpful conversations.

\end{document}